\documentclass{article}

\usepackage[utf8]{inputenc}
\usepackage[english]{babel}
\usepackage{hyperref}
\usepackage[utf8]{inputenc}
\usepackage{graphicx}
\usepackage{graphics}
\usepackage{amssymb}
\usepackage{amsmath}
\usepackage{amsthm}
\usepackage{tikz-cd}
\usepackage{mathrsfs}
\usepackage[colorinlistoftodos]{todonotes}
\usepackage{enumitem}
\usepackage{lmodern}
\usepackage[T1]{fontenc}
\usepackage{import}
\usepackage{wrapfig}
\usepackage{calc}
\usepackage[all]{xy}
\usepackage[margin=3.5cm]{geometry}
\usepackage{float}
\usepackage{mathtools}
\usepackage{subcaption}
\renewcommand{\epsilon}{\varepsilon}

\newtheorem{thm}{Theorem}[section]
\newtheorem{theorem}[thm]{Theorem}

\newtheorem{defn}[thm]{Definition}

\newtheorem{rmk}[thm]{Remark}
\newtheorem{lem}[thm]{Lemma}

\def\R{\mathbb{R}}

\def\C{\mathbb{C}}

\newtheorem{question}[thm]{Question}

\newcommand{\op}{\operatorname}

\def\R{\mathbb{R}}
\def\Z{\mathbb{Z}}

\def\C{\mathbb{C}}

\def\cal R{\mathcal R}

\title{Characterizing symplectic capacities on ellipsoids}
\author{Jean Gutt\thanks{jean.gutt@math.univ-toulouse.fr} \ and Vinicius G. B. Ramos\thanks{vgbramos@impa.br}}

\date{}

\begin{document}
\maketitle
\begin{abstract}
It is a long-standing conjecture that all symplectic capacities which are equal to the Gromov width for ellipsoids coincide on a class of convex domains in $\R^{2n}$. It is known that they coincide for monotone toric domains in all dimensions. In this paper, we study whether requiring a capacity to be equal to the $k^{th}$ Ekeland--Hofer capacity for all ellipsoids can characterize it on a class of domains. We prove that for $k=n=2$, this holds for convex toric domains, but not for all monotone toric domains. We also prove that for $k=n\ge 3$, this does not hold even for convex toric domains.
\end{abstract}
\section{Introduction}
Since Darboux's theorem it is known that all symplectic manifolds are locally ``the same''. In particular any symplectic invariant has to be of a global nature. This prompted the difficult quest in symplectic geometry of finding such invariants. The first nontrivial invariant was a so-called width defined by Gromov in \cite{gromov}. Inpired by Gromov's work, Ekeland and Hofer defined the concept of a symplectic capacity in \cite{EH}. If $X,X'$ are domains in $\R^{2n}$, a symplectic embedding from $X$ to $X'$ is a smooth embedding $\varphi:X \to X'$ such that $\varphi^*\omega=\omega$, where $\omega=\sum_{i=1}^n dx_i\wedge dy_i$ is the standard symplectic form on $\R^{2n}$. A \textbf{symplectic capacity} is a function $c$ which assigns to each subset in $\R^{2n}$ a number $c(X)\in[0,\infty]$ satisfying the following axioms:
\begin{description}
	\item{(Monotonicity)} If $X,X' \subset \R^{2n}$, and if there exists a symplectic embedding $X \hookrightarrow X'$, then $c(X) \le c(X')$.
	\item{(Conformality)} If $r$ is a positive real number then $c(rX) = r^2 c(X)$.	
\end{description}
Various examples of symplectic capacities have emerged such as the Hofer-Zehnder capacity $c_{\op{HZ}}$ defined in \cite{HZ} and the Viterbo capacity $c_{\op{SH}}$ defined in \cite{V}. There are also useful families of symplectic capacities parametrized by a positive integer $k$ including the Ekeland-Hofer capacities $c_k^{\op{EH}}$ defined in \cite{EH,EH2} using calculus of variations; the conjectured equal capacities $c_k^{\op{CH}}$ defined in \cite{GuH} using positive equivariant symplectic homology; and in the four-dimensional case, the ECH capacities $c_k^{\op{ECH}}$ defined in \cite{qech} using embedded contact homology. For more about symplectic capacities in general we refer to \cite{chls, schlenk} and the references therein. In view of all these different constructions and the desire for symplectic invariants to be defined axiomatically, a very natural question is the following.
\begin{question}\label{axiom}
	Is there an axiomatic characterization of some symplectic capacities?
\end{question}
One motivation for this question is a well-known conjecture which asserts that a certain normalization condition uniquely characterizes $c$ on the set of convex domains. This conjecture was recently disproved in the most general setting \cite{HO}, but it is still open for convex domains with $\Z/2$-symmetry. For $a_1,\dots,a_n\in(0,\infty]$, define the ellipsoid
\[E(a_1,\ldots,a_n):=\left\{z\in\C^n\;\big|\;\sum_{\substack{i=1\\ a_i\neq\infty}}^n\frac{\pi|z_i|^2}{a_i}\le 1\right\}.\]
A symplectic capacity $c$ is \textbf{ball normalized} if, for all $a_1\dots,a_n$, we have $c\big(E(a_1,\ldots,a_n)\big)=\min\{a_1,\ldots,a_n\}$. Notice that this condition is equivalent to the more usual one requiring that $c\big(B^{2n}(1)\big)=c\big(Z^{2n}(1)\big)=1$, where $B^{2n}(r)=E(r,\dots,r)$ and $Z^{2n}(r) = E(r,\infty,\dots,\infty)=B^2(r)\times \C^{n-1}$.  
This normalization does characterize a capacity on the set of monotone toric domains, as was proven in \cite{GHR, CGH}. This class is related but not the same as the class of convex sets. In particular, there are monotone toric domains which are not symplectomorphic to convex domains, see \cite{DGRZ}. Conversely, there exist convex domains not symplectomorphic to toric domains \cite{Hutzeta}. A natural step towards answering Question \ref{axiom} is the following question.
\begin{question}[Abbondandolo and Hutchings]\label{questionk}
	Does imposing the value of a symplectic capacity on all ellipsoids characterize it on a large family of domains?
\end{question}
\begin{rmk}
	The original question formulated by Hutchings is whether the axioms of \cite[Theorem 1.1]{GuH} uniquely characterize symplectic capacities for convex domains. These axioms are somewhat stronger than assuming that the capacities take certain values on ellipsoids. We will focus on Question \ref{questionk} in this paper.
\end{rmk}

The goal of this paper is to give some answers to Question \ref{questionk}. In order to state our results, we recall some definitions. Let $\mu:\R^{2n}\to[0,\infty)$ be the standard moment map, i.e., \[\mu(z_1,\dots,z_n)=(\pi|z_1|^2,\dots,\pi|z_n|^2).\] A toric domain is a set of the form $X_\Omega=\mu^{-1}(\Omega)$, where $\Omega\subset[0,\infty)^n$ is the closure of a non-empty relatively open set. Let $\partial_+\Omega=\partial \Omega\cap (0,\infty)^n$. We assume henceforth that $\partial_+\Omega$ is piecewise smooth.
\begin{defn}
A toric domain $X_\Omega$ is said to be:
\begin{itemize}
\item monotone if every outward-pointing normal vector\footnote{This definition includes domains for which $\partial_+\Omega$ is not smooth where the normal vector at a point is not uniquely defined, but where we can still define a normal cone.} $(v_1,\dots,v_n)$ to $\partial_+\Omega$ satisfies $v_i\ge 0$, for all $i$.
\item concave if $\Omega$ is compact and $\R_{\ge 0}^n\setminus\Omega$ is convex.
\item convex if \[\widehat{\Omega}:=\left\{(x_1,\ldots,x_n)\in\R^n \;\big|\; \left(|x_1|,\ldots,|x_n|\right)\in\Omega\right\}\] is compact and convex.
\end{itemize}
\end{defn}
\begin{rmk}
We note that every concave or convex toric domain is monotone. If $\Omega'\subset[0,\infty)^{n-1}$ is a monotone domain (i.e. $X_{\Omega'}$ is monotone) and $f:\Omega'\subset[0,\infty)^{n-1}\to\R$ is a smooth function such that $\frac{\partial f}{\partial x_i}\le 0$, for every $i$, then $X_{\op{Gr}(f)}$ is a monotone toric domain, where $\op{Gr(f)}$ is the region bounded above by the graph of $f$; $\op{Gr(f)}=\{(x,y)\in\Omega'\times[0,\infty)\,|\,y\leq f(x)\}$. Moreover, if $X_{\Omega'}$ and $f$ are simultaneously convex or concave, then $X_{\op{Gr}(f)}$ is concave or convex, respectively. We also observe that every monotone toric domain can be approximated (in the Hausdorff topology) by domains of the form $X_{\op{Gr}(f)}$. Finally we note that if $\partial_+\Omega$ is smooth and $X_\Omega$ is monotone, then $X_\Omega$ is dynamically convex, see \cite[Proposition 1.8]{GHR}.
\end{rmk}
 For $a_1,\dots,a_n>0$, let $N_k(a_1,\dots,a_n)$ denote the $k$-th smallest number in the multiset 
$\{m a_i\mid i=1,\dots,n \text{ and }m\in\Z_{>0}\}$. ``Multiset'' means that repetition of numbers is allowed. These correspond to the Ekeland--Hofer or the Gutt--Hutchings capacities of an ellipsoid:
\begin{equation}\label{eq:gh_ell}
c_k^{\rm EH}(E(a_1,\dots,a_n))=c_k^{\rm CH}(E(a_1,\dots,a_n))=N_k(a_1,\dots,a_n).
\end{equation}
\begin{defn}\label{def:kcap}
	A symplectic capacity $c$ is called {\bf k-normalized} if for every $a_1,\dots a_n>0$,
	\[
			c(E(a_1,\dots,a_n))=N_k(a_1,\dots,a_n).
	\]
	%We refer to \cite{GuH} for the definition and the explicit computation of $c^{\op{GH}}_k(E)$.
\end{defn}
To the best of our knowledge, this definition first appeared in \cite{BBLM}, and in \cite{ABE}.
\begin{rmk} Note that $N_1(a_1,\dots,a_n)=\min(a_1,\dots,a_n)$. So being $1$-normalized is equivalent to being ball normalized and thus by \cite{GHR,CGH}, 1-normalized capacities coincide for all monotone toric domains.
\end{rmk}

    The ``max'' and ``min'' k-normalized capacities are defined as:
    \[
        c_k^{\rm max}(X):=\inf\{N_k(a_1,\dots,a_n)\mid  X \hookrightarrow E(a_1,\dots,a_n),\text{ for some }a_1,\dots,a_n>0\}.
    \]
    \[
        c_k^{\rm min}(X):=\sup\{N_k(a_1,\dots,a_n)\mid E(a_1,\dots,a_n)\hookrightarrow X, \text{ for some }a_1,\dots,a_n>0\}.
    \]
    In fact, it is easy to see that if $c$ is a $k$-normalized capacity and $X\subset \R^{2n}$,
    \[c_k^{\rm min}(X)\le c(X)\le c_k^{\rm max}(X).\]
It follows that all k-normalized capacities coincide for $X$ if, and only if, $c_k^{\rm min}(X)=c_k^{\rm max}(X)$. One could then ask whether $k$-normalized capacities coincide for any class of convex domains. In fact, it was recently proven by Abbondandolo, Benedetti and Edtmair that this does hold for $n$-normalized capacities in a neighborhood of the round ball.
\begin{theorem}[\cite{ABE}]
There exists a $C^2$-neighborhood of the ball $B^{2n}(1)$ in which all $n$-normalized capacities coincide.
\end{theorem}

Our first result is that this holds in dimension 4 for all convex toric domains, but not for all monotone domains.
\begin{theorem}\label{thm:dim2}
If $X\subset \R^4$ is a convex toric domain, then $c_2^{\rm min}(X)=c_2^{\rm max}(X)$. Moreover, there exists a concave toric domain $V\subset \R^4$ such that $c_2^{\rm min}(V)< c_2^{\rm max}(V)$.
\end{theorem}

Our second result says that $k=2$ is unusual. Let $P(a_1,\dots,a_n)$ be the polydisk defined by
\[P(a_1,\dots,a_n)=\left\{z\in\C^n\mid \pi|z_i|^2\le a_i\right\}.\]
\begin{theorem}\label{thm:hdim}
	For all  $a_1,\dots,a_n>0$, \[c_2^{\rm min}(P(a_1,\dots,a_n))=c_2^{\rm max}(P(a_1,\dots,a_n)).\]
	If $k\ge \max(n,3)$ and $P(1,\dots,1)\subset \C^n$, then \[c_k^{\rm min}(P(1,\dots,1))<c_k^{\rm max}(P(1,\dots,1)).\]
\end{theorem}
It follows from Theorem \ref{thm:hdim} that $n$-normalized capacities cannot coincide for all convex sets if $n\ge 3$.

The rest of the paper is organized in three sections. In Section \ref{sec:conv2d} we prove the first part of Theorem \ref{thm:dim2}. In Section \ref{sec:example}, we construct a concave toric domain for which 2-normalized capacities do not coincide. Finally in Section \ref{sec:knorm}, we prove Theorem \ref{thm:hdim}.

\noindent{\bf Acknowledgements}
	We would like to thank Alberto Abbondandolo and Michael Hutchings for pointing us the question about $k$-normalized capacities. We are grateful to Alberto Abbondandolo, Oliver Edtmair and Gabriele Benedetti for sharing a preliminary version of their work \cite{ABE} with us. We would also like to particularly thank Richard Hind for his insight on embedding ellipsoids into polydisks in higher dimensions and the argument of the proof of the first part of Theorem \ref{thm:hdim}.
	
	J.G. was partially supported by the ANR LabEx CIMI (grant ANR-11-LABX-0040) within the French State Programme ``Investissements d’Avenir'' and by the ANR COSY (ANR-21-CE40-0002) grant.

V.G.B.R.\ is grateful for the hospitality of the Institute for Advanced Study, where part of this work was completed. V.G.B.R.\ was partially supported by NSF grant DMS-1926686, FAPERJ grant JCNE E-26/201.399/2021 and a Serrapilheira Institute grant.

\section{2-normalized capacities for 4-dimensional convex toric domains}\label{sec:conv2d}
In this section we prove the first part of Theorem \ref{thm:dim2}, namely that for a convex toric domain $X_\Omega\subset \R^4$, we have \begin{equation}\label{eq:conv2d}
c_2^{\rm min}(X_\Omega)=c_2^{\rm max}(X_\Omega).
\end{equation}

Before proving \eqref{eq:conv2d} in all generality, we prove the special case of a polydisk.
\begin{lem}\label{lem:polydisk}
For any 2-normalized symplectic capacity $c$, we have $c(P(a,b))=2\min(a,b)$.
\end{lem}
\begin{proof}
 We can assume without loss of generality that $a\le b$. Let $\varepsilon>0$ such that $\varepsilon<b/2$. It follows from \cite[Theorem 1.3]{FM} that $\op{int}(E(a,2a))\hookrightarrow P(a,a)$. So
    \begin{equation}\label{eq:incl_pb1}
      \op{int}(E(a,2a))\hookrightarrow P(a,a)\subset P(a,b)\subset E\left(a+\epsilon,\frac{(a+\epsilon)b}{\epsilon}\right)
    \end{equation}
    The last inclusion above is a simple analytic geometry exercise.  
%        \begin{center}\begin{tikzpicture}
%         \fill[orange] (0,0)--(2,0)--(0,1)--(0,0);
%        \draw (0,0)--(3,0);
%        \draw (0,0)--(0,3);
%        \draw (2,0) node[below]{$2$};
%        \draw (0,1) node[left]{$1$};
%        \fill[orange] (5,0)--(6,0)--(6,1)--(5,1)--(5,0);
%        \draw (5,0)--(12,0);
%        \draw (5,0)--(5,3);
%        \draw (5,1)--(6.5,1);
%        \draw (6.5,1)--(6.5,0);
%         \draw (5,1) node[left]{$1$};
%         \draw (6.5,0) node[below]{$a$};
%         \draw (5,1.3)--(11.5,0);
%         \draw (5,1.3) node[left]{$1+\varepsilon$};
%         \draw (11.5,0) node[below]{$\frac{(1+\varepsilon)a}{\varepsilon}$};
%         \draw (6,0) node[below]{$1$};
%         \draw (3.7,0.4) node{$\longhookrightarrow$};
%    \end{tikzpicture}
%    \end{center}
    If $c$ is any 2-normalized capacity, it follows from Definition \ref{def:kcap}, \eqref{eq:incl_pb1} and \eqref{eq:gh_ell} that
  \[2a=c\left(E(a,2a)\right)\le c(P(a,b))\le c\left(E\left(a+\epsilon,\frac{(a+\epsilon)b}{\epsilon}\right)\right)=2(a+\varepsilon).
\]
Taking the limit as $\varepsilon\to 0$, we conclude that $c(P(a,b))=2a$.
\end{proof}
%\begin{rmk}
%    It is also possible to prove the same statement for the cylinder $Z(a)$ using an exhaustion argument. It is then easy to show that any 2-normalized capacity of the polydisk $P(a,b)$ must be equal to $2\min\{a,b\}$.
%\end{rmk}

Now, let $X_\Omega$ be a 4-dimensional convex toric domain whose moment map image intercepts with the axes are $a$ and $b$. We can assume without loss of generality that $a\le b$. Since $X_\Omega\subset P(a,b)$, it follows from Lemma \ref{lem:polydisk} that $c_2^{\rm max}(X_\Omega)\le 2a$. Now let $w$ be the minimum of $r>0$ such that $X_\Omega\subset B(r)$. Since $X_\Omega\subset B(w)$, it follows from Definition \ref{def:kcap} that $c_2^{\rm max}(X_\Omega)\le w$. Therefore
\begin{equation}\label{eq:c2max}
c_2^{\rm max}(X_\Omega)\le \min(2a,w).
\end{equation}
Now we observe that there exists a point $(x,y)\in\partial_+\Omega$ such that $x+y=w$. Let $Q$ denote the convex hull of $(0,0)$, $(a,0)$, $(x,y)$ and $(0,b)$. Since $\Omega$ is convex, $X_Q\subset X_\Omega$. It follows from \cite[Proposition 3.5]{GU} that $E(a,w)\hookrightarrow X_Q$ and so $E(a,w)\hookrightarrow X_\Omega$. Using Definition \ref{def:kcap} again, we obtain
\begin{equation}\label{eq:c2min}
c_2^{\rm min}(X_\Omega)\ge \min(2a,w).
\end{equation}
Combining \eqref{eq:c2max} and \eqref{eq:c2min}, we conclude that $c_2^{\rm min}(X_\Omega)=c_2^{\rm max}(X_\Omega)$.

\section{A family of concave toric domains}\label{sec:example}

We now prove the second statement in Theorem \ref{thm:dim2}. In fact, we will construct a family of concave toric domains $V_\epsilon$ such that $c_2^{\rm min}(V_\epsilon)<c_2^{\rm max}(V_\epsilon)$ for $\epsilon$ sufficiently small. For $\epsilon\in (0,1/2)$, let $V_\epsilon=\mu^{-1}(Q_\epsilon)$, where $Q_\epsilon$ is the quadrilateral with vertices $(0,0),\,(0,1),\,(\epsilon,\epsilon),\,(1,0)$, see Figure \ref{fig:veps}(a).

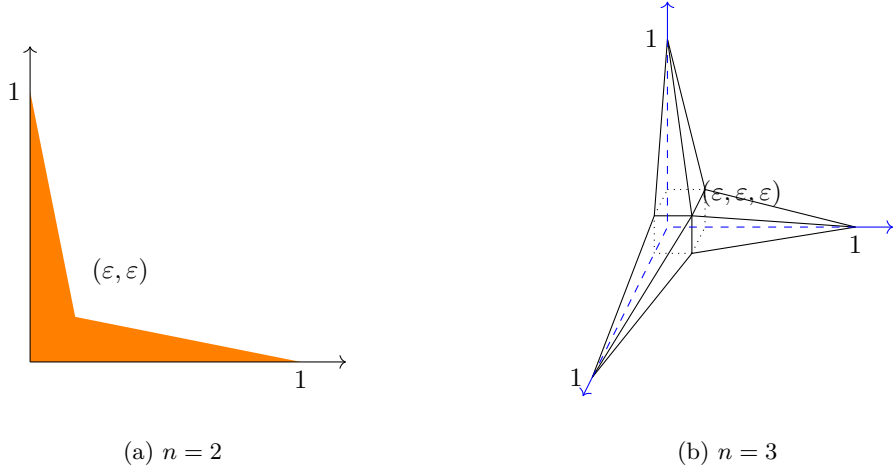
\begin{figure} 
\begin{subfigure}{0.5\textwidth}
    \begin{center}\begin{tikzpicture}[scale=1.2]
         \fill[orange] (0,0)--(3,0)--(0.5,0.5)--(0,3)--(0,0);
        \draw[->] (0,0)--(3.5,0);
        \draw[->] (0,0)--(0,3.5);
        \draw (1,1) node {$(\epsilon,\epsilon)$};
        \draw (3,0) node[below]{1};
        \draw (0,3) node[left]{1};
    \end{tikzpicture}
    \end{center}\subcaption{$n=2$}
    \end{subfigure}
    \begin{subfigure}{0.5\textwidth}
        \begin{center}\begin{tikzpicture}[scale=0.5]
%\draw[color=blue] (-1,0) -- (0,0);
\draw[->,color=blue] (5,0) -- (6,0);
%\draw[color=blue] (0,-1) -- (0,0);
\draw[->,color=blue] (0,5) -- (0,6);
%\draw[color=blue] (0.25,0.5)--(0,0);
\draw[->,color=blue] (-2,-4)--(-2.25,-4.5);
\draw (0,5)--(1,1)--(5,0);
\draw (5,0)--(0.65,-0.7)--(-2,-4);
\draw (-2,-4)--(-0.35,0.3)--(0,5);
\draw (0.65,-0.7)--(0.65,0.3);
\draw (0.65,0.3)--(1,1);
\draw (0.65,0.3)--(-0.35,0.3);
\draw (-2,-4)--(0.65,0.3);
\draw (0,5)--(0.65,0.3);
\draw (0.65,0.3)--(5,0);
\draw[dotted] (1,1)--(1,0);
\draw[dotted] (0.65,-0.7)--(1,0);
\draw[dotted] (0.65,-0.7)--(-0.35,-0.7);
\draw[dotted] (-0.35,-0.7)--(-0.35,0.3);
\draw[dotted] (-0.35,0.3)--(0,1);
\draw[dotted] (0,1)--(1,1);
\draw[dashed, color=blue] (-2,-4)--(0,0);
\draw[dashed, color=blue] (5,0)--(0,0);
\draw[dashed, color=blue] (0,5)--(0,0);
\draw (5,0) node[below]{1};
\draw (0,5) node[left]{1};
\draw (-2,-4) node[left]{1};
\draw (0.65,0.3) node[above right]{$(\epsilon,\epsilon,\epsilon)$};
\end{tikzpicture}
\end{center}\subcaption{$n=3$}
\end{subfigure}
    \caption{The domain $V_\epsilon$}\label{fig:veps}
 \end{figure}   
The following theorem is enough to prove the second statement in Theorem \ref{thm:dim2} although it proves quite a bit more. 
\begin{theorem}\label{thm:veps}
 Let $\epsilon\in(0,1/2)$. Then $c_2^{\rm min}(V_\epsilon)=c_2^{\rm max}(V_\epsilon)$ if, and only if $\epsilon\ge 2/9$.
\end{theorem}    
\begin{proof}
Let $\epsilon\in[2/9,1/2)$. We will show that $c_2^{\rm min}(V_\epsilon)=c_2^{\rm max}(V_\epsilon)$.

First suppose that $\epsilon\in[1/3,1/2)$. Then 
\begin{equation*}\label{eq:inc}
E(1/2,1)\subset V_\epsilon\subset B^4(1)=E(1,1).
\end{equation*}
Since $N_2(1/2,1)=N_2(1,1)=1$, it follows that
\begin{equation*}
c_2^{\rm max}(V_\epsilon)\le N_2(1,1)=N_2(1/2,1)\le c_2^{\rm min}(V_\epsilon).
\end{equation*}
Therefore $c_k^{\rm min}(V_\epsilon)=c_k^{\rm max}(V_\epsilon)$.

Before proceeding to the case when $\epsilon\in[2/9,1/3)$, we now recall how one can associate to a concave toric domain $X_\Omega$ a natural ball packing $\bigsqcup_{j=1}^\infty \op{int}(B^4(w_j))\hookrightarrow X_\Omega$. For a more thorough explanation, see \cite{concave}. Let $T(w)\subset \R_{\ge 0}^2$ be the triangle with vertices $(0,0)$, $(w,0)$ and $(0,w)$. So $\mu^{-1}(T(w))=B^4(w)$. Let $w_1$ the supremum of $w$ such that $T(w)\subset \Omega$. So $\Omega\setminus T(w_1)$ has of two connected components $\Omega_1$ and $\Omega_1'$. We translate $\Omega_1$ and $\Omega_1'$ so that the corners lie at the origin and we apply the linear transformations $\begin{bmatrix}1&1\\0&1\end{bmatrix}$ and $\begin{bmatrix}1&0\\1&1\end{bmatrix}$, respectively. By taking the closure of the sets above, we obtain two domains $\widetilde{\Omega}_1$ and $\widetilde{\Omega}_1'$ in $\R^2_{\ge 0}$. Now we define the next weights $w_2$ and $w_3$ to be the suprema of $w$ such that $T(w)\subset \widetilde{\Omega}_1$ and $T(w)\subset \widetilde{\Omega}_1'$, respectively, arranged in decreasing order. We continue this process by induction and we obtain a decreasing sequence $w_1,w_2,\dots$. It follows from the ``Traynor trick'' (\cite[Proposition 5.2]{traynor}) followed by a transformation in $SL(2,\Z)$, as explained in \cite[Lemma 1.8]{concave}, that 
\[\bigsqcup_{j=1}^\infty \op{int}(B(w_j))\hookrightarrow X_\Omega.\] In \cite{CG}, Cristofaro-Gardiner proved that \begin{equation}\label{eq:bp_emb}
\op{int}(X_\Omega)\hookrightarrow B(a)\iff \bigsqcup_{j=1}^\infty \op{int}(B^4(w_j))\hookrightarrow B^4(a).\end{equation}

Now suppose that $\epsilon\in [2/9,1/3)$. Let $w_1\ge w_2\ge w_3\ge \dots$ be the weights of $V_\epsilon$ as defined above. It follows from a simple calculation that $w_1=2\epsilon$ and $w_2=w_3=\epsilon$. The triangles $T(w_1)$, $T(w_2)$ and $T(w_3)$ correspond to the red and the two blue triangles in Figure \ref{fig:bp}(a), respectively. Moreover the domains $\Omega_2:=\op{cl}(\widetilde{\Omega}_1\setminus B^4(w_2)) $ and $\Omega_2':=\op{cl}(\widetilde{\Omega}'_1\setminus B^4(w_3))$ are triangles which are affinely equivalent under $SL(2,\Z)$ to the yellow triangles in Figure \ref{fig:bp}(a). It follows that $\Omega_2$ and $\Omega'_2$ are affinely equivalent to a right triangle $T(\epsilon,1-3\epsilon)$ with sides $\epsilon$ and $1-3\epsilon$. Note that $X_{T(\epsilon,1-3\epsilon)}=E(\epsilon,1-3\epsilon)$. It follows that \begin{equation}\label{eq:bp_oddeven}
\bigsqcup_{k=2}^\infty \op{int}(B^4(w_{2k}))\hookrightarrow \op{int}(E(\epsilon,1-3\epsilon)),\text{ and }\bigsqcup_{k=2}^\infty \op{int}(B^4(w_{2k+1}))\hookrightarrow \op{int}(E(\epsilon,1-3\epsilon)).
\end{equation}
Now we can find an explicit packing of the interiors of affine copies of $T(2\epsilon)$, two copies of $T(\epsilon)$ and two copies of $T(\epsilon,1-3\epsilon)$ into $T(3\epsilon)$, as shown in Figure \ref{fig:bp}(b). Here we use the fact that $1-3\epsilon<6\epsilon-1$, which is equivalent to $\epsilon\ge 2/9$. It follows from the ``Traynor trick'' that
\begin{equation}\label{eq:bp_all}
\op{int}(B^4(w_1))\sqcup \op{int}(B^4(w_2))\sqcup \op{int}(B^4(w_3))\sqcup \op{int}(X_{T(\epsilon)})\sqcup \op{int}(X_{T(\epsilon)})\hookrightarrow T(3\epsilon).
\end{equation}
Using \eqref{eq:bp_emb}, \eqref{eq:bp_oddeven} and \eqref{eq:bp_all}, we conclude that
\[V_\epsilon\hookrightarrow B^4(3\epsilon).\]
Moreover, since $\epsilon<1/3$, it follows that $E(3\epsilon/2,3\epsilon)\subset V_\epsilon$. Therefore
\[c_2^{\rm max}(V_\epsilon)\le N_2(3\epsilon,3\epsilon)=N_2(3\epsilon/2,3\epsilon)\le c_2^{\rm min}(V_\epsilon).\]
Hence $c_2^{\rm min}(V_\epsilon)=c_2^{\rm max}(V_\epsilon)$.

\begin{figure}
\centering
\begin{subfigure}{0.4\textwidth}
%% Creator: Inkscape 1.1.2 (0a00cf5339, 2022-02-04), www.inkscape.org
%% Accompanies image file 'concave.pdf' (pdf, eps, ps)
\begingroup%
  \makeatletter%
  \providecommand\color[2][]{%
    \errmessage{(Inkscape) Color is used for the text in Inkscape, but the package 'color.sty' is not loaded}%
    \renewcommand\color[2][]{}%
  }%
  \providecommand\transparent[1]{%
    \errmessage{(Inkscape) Transparency is used (non-zero) for the text in Inkscape, but the package 'transparent.sty' is not loaded}%
    \renewcommand\transparent[1]{}%
  }%
  \providecommand\rotatebox[2]{#2}%
  \newcommand*\fsize{\dimexpr\f@size pt\relax}%
  \newcommand*\lineheight[1]{\fontsize{\fsize}{#1\fsize}\selectfont}%
  \ifx\svgwidth\undefined%
    \setlength{\unitlength}{142.45516837bp}%
    \ifx\svgscale\undefined%
      \relax%
    \else%
      \setlength{\unitlength}{\unitlength * \real{\svgscale}}%
    \fi%
  \else%
    \setlength{\unitlength}{\svgwidth}%
  \fi%
  \global\let\svgwidth\undefined%
  \global\let\svgscale\undefined%
  \makeatother%
  \begin{picture}(1,1.06785874)%
    \lineheight{1}%
    \setlength\tabcolsep{0pt}%
    \put(0,0){\includegraphics[width=\unitlength,page=1]{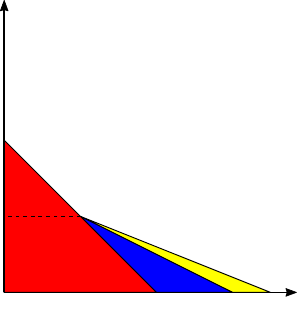}}%
    \put(0.22933725,0.01476997){\color[rgb]{0,0,0}\makebox(0,0)[lt]{\lineheight{1.25}\smash{\begin{tabular}[t]{l}$\epsilon$\end{tabular}}}}%
    \put(0.51493678,0.01312296){\color[rgb]{0,0,0}\makebox(0,0)[lt]{\lineheight{1.25}\smash{\begin{tabular}[t]{l}$2\epsilon$\end{tabular}}}}%
    \put(0.73588456,0.02102926){\color[rgb]{0,0,0}\makebox(0,0)[lt]{\lineheight{1.25}\smash{\begin{tabular}[t]{l}$3\epsilon$\end{tabular}}}}%
    \put(0,0){\includegraphics[width=\unitlength,page=2]{concave.pdf}}%
    \put(0.89080779,0.02041126){\color[rgb]{0,0,0}\makebox(0,0)[lt]{\lineheight{1.25}\smash{\begin{tabular}[t]{l}$1$\end{tabular}}}}%
  \end{picture}%
\endgroup%
\caption{}
\end{subfigure}\qquad
\begin{subfigure}{0.4\textwidth}
%% Creator: Inkscape 1.1.2 (0a00cf5339, 2022-02-04), www.inkscape.org
%% Accompanies image file 'ballpacking.pdf' (pdf, eps, ps)
\begingroup%
  \makeatletter%
  \providecommand\color[2][]{%
    \errmessage{(Inkscape) Color is used for the text in Inkscape, but the package 'color.sty' is not loaded}%
    \renewcommand\color[2][]{}%
  }%
  \providecommand\transparent[1]{%
    \errmessage{(Inkscape) Transparency is used (non-zero) for the text in Inkscape, but the package 'transparent.sty' is not loaded}%
    \renewcommand\transparent[1]{}%
  }%
  \providecommand\rotatebox[2]{#2}%
  \newcommand*\fsize{\dimexpr\f@size pt\relax}%
  \newcommand*\lineheight[1]{\fontsize{\fsize}{#1\fsize}\selectfont}%
  \ifx\svgwidth\undefined%
    \setlength{\unitlength}{161.67181798bp}%
    \ifx\svgscale\undefined%
      \relax%
    \else%
      \setlength{\unitlength}{\unitlength * \real{\svgscale}}%
    \fi%
  \else%
    \setlength{\unitlength}{\svgwidth}%
  \fi%
  \global\let\svgwidth\undefined%
  \global\let\svgscale\undefined%
  \makeatother%
  \begin{picture}(1,0.91666572)%
    \lineheight{1}%
    \setlength\tabcolsep{0pt}%
    \put(0,0){\includegraphics[width=\unitlength,page=1]{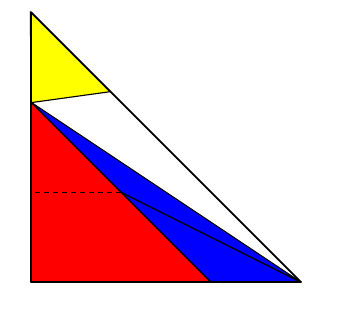}}%
    \put(0.31491171,0.01006533){\color[rgb]{0,0,0}\makebox(0,0)[lt]{\lineheight{1.25}\smash{\begin{tabular}[t]{l}$1-3\epsilon$\end{tabular}}}}%
    \put(0.63703137,0.01226854){\color[rgb]{0,0,0}\makebox(0,0)[lt]{\lineheight{1.25}\smash{\begin{tabular}[t]{l}$6\epsilon-1$\end{tabular}}}}%
    \put(0.00738979,0.33339787){\color[rgb]{0,0,0}\makebox(0,0)[lt]{\lineheight{1.25}\smash{\begin{tabular}[t]{l}$\epsilon$\end{tabular}}}}%
    \put(0.00344849,0.59979356){\color[rgb]{0,0,0}\makebox(0,0)[lt]{\lineheight{1.25}\smash{\begin{tabular}[t]{l}$2\epsilon$\end{tabular}}}}%
    \put(-0.00256602,0.86663371){\color[rgb]{0,0,0}\makebox(0,0)[lt]{\lineheight{1.25}\smash{\begin{tabular}[t]{l}$3\epsilon$\end{tabular}}}}%
    \put(0.84025407,0.02332821){\color[rgb]{0,0,0}\makebox(0,0)[lt]{\lineheight{1.25}\smash{\begin{tabular}[t]{l}$3\epsilon$\end{tabular}}}}%
    \put(0,0){\includegraphics[width=\unitlength,page=2]{ballpacking.pdf}}%
  \end{picture}%
\endgroup%
\caption{}
\end{subfigure}
\caption{Ball packings}\label{fig:bp}
\end{figure}

Now we suppose that $\epsilon<2/9$. Using the explicit formula for $c^{\op{CH}}_2$ for any concave toric domain \cite[Theorem 1.14]{GuH}, we have that
	\[
	    	c^{\op{CH}}_2(V_\epsilon)=3\epsilon.
	\]
	On the other hand, we claim that, $c_2^{\rm max}(V_\epsilon)>3\epsilon$.
	To prove the claim, it is enough to show that there exists $\delta>0$ such that for any ellipsoid $E(a,b)$ such that $V_\epsilon\hookrightarrow E(a,b)$ we have $N_2(a,b)\big)\ge  3\epsilon+\delta$. We assume, without loss of generality, that $a\leq b$.
    Since $B^4(2\epsilon)\subset V_\epsilon\hookrightarrow E(a,b)$, it follows that
    \[
        2\epsilon=N_1(2\epsilon,2\epsilon)\leq N_1(a,b)=a.
    \]
    We also have
    \[
        \epsilon=\op{Vol}(V_\epsilon)\leq \op{Vol}\big(E(a,b)\big)=\frac{ab}{2}\leq\frac{b^2}{2}.
    \]
    Thus $b\geq\sqrt{2\epsilon}>3\epsilon+\delta$, for some small $\delta>0$, since $\epsilon<2/9$. Moreover,
    \begin{equation*}\label{eq:c2gh}
        N_2(a,b)=\min\{2a,b\}\ge \min(4\epsilon, 3\epsilon+\delta)\ge 3\epsilon+\delta.
        \end{equation*}
        It follows that
        \[c_2^{\rm max}(V_\epsilon)\ge 3\epsilon+\delta=c_2^{\rm 
        CH}(V_\epsilon)+\delta>c_2^{\rm min}(V_\epsilon).\]
\end{proof}

\begin{rmk}
Using a similar argument to \cite[Theorem 1.2]{DGRZ}, it can be shown that $V_\epsilon$ is not symplectomorphic to a convex set for $\epsilon$ sufficiently small.
\end{rmk}

\begin{rmk}
The family $V_\epsilon$ can be generalized to higher dimensions, see Figure \ref{fig:veps}(b). Using the same argument as in the proof of Theorem \ref{thm:veps}, we can show that if $\epsilon<n!/(2n-1)^n$, then $c_n^{\min}(V_\epsilon)<c_n^{\max}(V_\epsilon)$. However, we do not know what happens when $\epsilon\ge n!/(2n-1)^n$ since the techniques used in this case do not generalize to higher dimensions.
\end{rmk}

\section{Proof of Theorem \ref{thm:hdim}}\label{sec:knorm}
We start by the case $k=2$. It follows from the proof of Theorem \ref{thm:dim2} that we just need to show that in $\C^{n}$
\begin{equation}\label{eq:embed}
	\op{int}E(1,2,\ldots,2)\hookrightarrow P(1,\ldots,1)
\end{equation}
We shall proceed by induction on the dimension and use twice the family method from \cite[Section 2.1]{BH}. For ease of readability, we shall add the complex dimension of the ellipsoid in subscript.  By \cite[Theorem 1.3]{FM} the embedding \eqref{eq:embed} exist for $n=2$. We regard $E_n(1,2,\ldots,2)$ as a fibration over $E_n(1,2,\ldots,2)\cap\{z_n-plane\}$ where the fiber over a point $(0,\ldots,0,z_n)$ with $\pi|z_n|^2=r$ is an ellipsoid $(1-\frac{r}{2})E_{n-1}(1,2,\ldots,2)\subset\C^{n-1}$. By induction, this fiber embeds in $(1-\frac{r}{2})P_{n-1}(1,\ldots,1)\subset P_{n-1}(1, 1-\frac{r}{2},\ldots,1-\frac{r}{2})$.
We change the viewpoint and see it as a fibration over the $z_1$-plane. The family method then says we can do this simultaneously for all fibers and get an embedding
\[
	E_n(1,2,\ldots,2)\hookrightarrow D(1)\times E_{n-1}(1,2,\ldots,2)	
\]
We now use the family method a second time, together with our induction assumption to obtain an embedding
\[
	E_n(1,2,\ldots,2)\hookrightarrow D(1)\times P_{n-1}(1,\ldots,1)=P_n(1,\ldots,1).
\] 

For the second part of Theorem \ref{thm:hdim}, we fix $k,n\ge 2$ such that $k\ge \max(n,3)$. We shall prove for the polydisk $P(1,\dots,1)\subset\C^n$, we have $c_k^{\rm min}(P(1,\dots,1))<c_k^{\rm max}(P(1,\dots,1))$.

We have by \cite[Theorem 1.12]{GuH}
\[
	c_k^{CH}(P(1,\dots,1))=k.
\]
We suppose that $c_k^{\rm min}(P(1,\dots,1))=k$. Let $0<\epsilon<1$. Then there exist $0<a_1\le\dots\le a_n$ such that 
\begin{eqnarray}
E(a_1,\dots,a_n)\hookrightarrow& P(1,\dots,1),\label{eq:emb}\\
N_k(a_1,\dots,a_n)\ge& k-\epsilon.\label{eq:ineq1}
\end{eqnarray}

We first suppose that $a_2\le (k-1) a_1$. Under this condition, it follows from the definition of $N_k$ that
\begin{equation}\label{eq:ineq2}
N_k(a_1,\dots,a_n)\le (k-1)a_1.
\end{equation}
From \eqref{eq:emb} we obtain $a_1=c_1^{\rm CH}(E(a_1,\dots,a_n))\le c_1^{\rm CH}(P(1,\dots,1))=1$. So it follows from \eqref{eq:ineq2} that
\[N_k(a_1,\dots,a_n)\le k-1,\]
contradicting \eqref{eq:ineq1}.

So we may assume that $a_2> (k-1) a_1$. Under this condition, the definition of $N_k$ implies that
\begin{equation}\label{eq:ineq3}
N_k(a_1,\dots,a_n)\le \min(k a_1,a_2).
\end{equation}
Combining \eqref{eq:ineq1} and \eqref{eq:ineq3}, we obtain
\begin{equation}\label{eq:ineq4}
a_1\ge \frac{k-\epsilon}{k},\quad a_i\ge k-\epsilon\text{ for }i\ge 2.
\end{equation}
Since a symplectic embedding is volume preserving, it follows from \eqref{eq:emb} and \eqref{eq:ineq4} that
 \begin{equation}\frac{(k-\epsilon)^{n}}{k\cdot n!}\le \frac{a_1\dots a_n}{n!}=\op{vol}(E(a_1,\dots,a_n))\le \op{vol}(P(1,\dots,1))=1.\end{equation}
Taking the limit as $\epsilon\to 0$ in \eqref{eq:ineq4}, we obtain
\begin{equation}\label{eq:lim}
\frac{k^n}{k\cdot n!}\le 1.
\end{equation}
But \eqref{eq:lim} does not hold under our assumptions. In fact, it is obvious that \eqref{eq:lim} is false for $k=3$ and $n=2$. Now if $k\ge n\ge 3$, then
\[\frac{k^n}{k\cdot n!}=\frac{k^{n-1}}{n!}\ge \frac{n^{n-2}}{(n-1)!}\ge \frac{n}{n-1}>1,\]
contradicting \eqref{eq:lim}. We conclude that our original assumption is false, and hence \[c_k^{\rm min}(P(1,\dots,1))<k=c_k^{\rm GH}(P(1,\dots,1))\le c_k^{\rm max}(P(1,\dots,1)).\]

\begin{rmk}
The conditions on $k$ and $n$ for which the conclusion of Theorem \ref{thm:hdim} holds can be relaxed. In particular, it holds for $1\le k
< n$ for which \eqref{eq:lim} is false.
\end{rmk}

%Let $E(a,b,c)$ be an ellipsoid such that $E(a,b,c)\underset{s}{\hookrightarrow}P(1,1,1)$. We can assume without loss of generality that $a\leq b\leq c$.
%Since there exist an embedding, we have
%\[
%	a=c_1^{GH}(E(a,b,c))\leq c_1^{GH}(P(1,1,1))=1.
%\]
%By \cite{GuH}[Theorem 1.12], we have
%\[
%	c_3^{GH}(E(a,b,c))=\begin{cases}
%%		\min\{a,b,c\}&\textrm{if }c\leq 2a\\
%		\min\{2a,2b,c\}&\textrm{if }b\leq 2a\leq c\\
%		\min\{3a,b,c\}&\textrm{if } 2a\leq b
%	\end{cases}
%\]
%In the first case, we have $\min\{2a,2b,c\}\leq2$ since $a\leq1$.
%In the second case, assume by contradiction that $\min\{3a,b,c\}\geq2.7$. This implies that $a\geq0.9$ and $b\geq2.7$ and $c\geq2.7$. Therefore $Vol(E(a,b,c)) = \frac{abc}{6}>1$ which is a contradiction with
%\[
%	Vol(E(a,b,c))\leq Vol(P(1,1,1))=1.
%\]
%This conclude the proof.
%
%The same argument,  of course, generalizes to any dimension and give the following:
%\begin{prop}
%	In $\C^n$ with $n\geq3$, the $n$-normalized capacities \emph{do not} coincide on the convex domain $P(1,\ldots,1)$.
%\end{prop}

\bibliographystyle{alpha}
\bibliography{bibliography.bib}

\newcommand{\etalchar}[1]{$^{#1}$}
\begin{thebibliography}{CCGF{\etalchar{+}}14}

\bibitem[ABE]{ABE}
Alberto Abbondandolo, Gabriele Benedetti, and Oliver Edtmair.
\newblock Symplectic capacities of domains close to the ball and banach--mazur
  geodesics in the space of contact forms.
\newblock {\em Preprint arXiv:2312.07363}.

\bibitem[BBLM]{BBLM}
L~Baracco, O~Bernardi, C~Lange, and M~Mazzucchelli.
\newblock On the local maximizeers of higher capacity ratios.
\newblock {\em to appear in Annali della Scuola Normale Superiore di Pisa,
  Classe di Scienze}.

\bibitem[BH11]{BH}
Olguta Buse and Richard Hind.
\newblock Symplectic embeddings of ellipsoids in dimension greater than four.
\newblock {\em Geom. Topol.}, 15(4):2091--2110, October 2011.

\bibitem[CCGF{\etalchar{+}}14]{concave}
K.~Choi, D.~Cristofaro-Gardiner, D.~Frenkel, M.~Hutchings, and V.~Ramos.
\newblock Symplectic embeddings into four-dimensional concave toric domains.
\newblock {\em J. Topol.}, 7:1054--1076, 2014.

\bibitem[CG19]{CG}
Dan Cristofaro-Gardiner.
\newblock Symplectic embeddings from concave toric domains into convex ones.
\newblock {\em J. Differential Geom.}, 112(2):199--232, 2019.

\bibitem[CGH23]{CGH}
Daniel Crisfaro-Gadiner and Richard Hind.
\newblock On the agreement of symplectic capacities in high dimension.
\newblock {\em arXiv:2307.12125}, 2023.

\bibitem[CHLS07]{chls}
Kai Cieliebak, Helmut Hofer, Janko Latschev, and Felix Schlenk.
\newblock Quantitative symplectic geometry.
\newblock In {\em Dynamics, ergodic theory, and geometry}, volume~54 of {\em
  Math. Sci. Res. Inst. Publ.}, pages 1--44. Cambridge Univ. Press, Cambridge,
  2007.

\bibitem[DGRZ23]{DGRZ}
Julien Dardennes, Jean Gutt, Vinicius G.~B. Ramos, and Jun Zhang.
\newblock Coarse distance from dynamically convex to convex.
\newblock {\em arXiv:2309.10912}, 2023.

\bibitem[EH89]{EH}
Ivar Ekeland and Helmut Hofer.
\newblock Symplectic topology and {H}amiltonian dynamics.
\newblock {\em Math. Z.}, 200(3):355--378, 1989.

\bibitem[EH90]{EH2}
Ivar Ekeland and Helmut Hofer.
\newblock Symplectic topology and {H}amiltonian dynamics. {II}.
\newblock {\em Math. Z.}, 203(4):553--567, 1990.

\bibitem[FM15]{FM}
David Frenkel and Dorothee M\"{u}ller.
\newblock Symplectic embeddings of 4-dim ellipsoids into cubes.
\newblock {\em J. Symplectic Geom.}, 13(4):765--847, 2015.

\bibitem[GH18]{GuH}
Jean Gutt and Michael Hutchings.
\newblock Symplectic capacities from positive {S}1--equivariant symplectic
  homology.
\newblock {\em Algebr. Geom. Topol.}, 18(6):3537--3600, 2018.

\bibitem[GHR22]{GHR}
Jean Gutt, Michael Hutchings, and Vinicius G.~B. Ramos.
\newblock Examples around the strong {V}iterbo conjecture.
\newblock {\em J. Fixed Point Theory Appl.}, 24(2):Paper No. 41, 22, 2022.

\bibitem[Gro85]{gromov}
Mikhail Gromov.
\newblock Pseudoholomorphic curves in symplectic manifolds.
\newblock {\em Invent. Math.}, 82:307--347, 1985.

\bibitem[GU19]{GU}
Jean Gutt and Michael Usher.
\newblock Symplectically knotted codimension-zero embeddings of domains in
  $\mathbb{R}^4$.
\newblock {\em Duke Mathematical Journal}, 168(12):2299--2363, 2019.

\bibitem[HKO]{HO}
Pazit Haim-Kislev and Yaron Ostrover.
\newblock A {C}ounterexample to {V}iterbo's conjecture.
\newblock {\em arXiv:2405.16513}.

\bibitem[Hut]{Hutzeta}
Michael Hutchings.
\newblock Zeta functions of dynamically tame liouville domains.
\newblock {\em arXiv:2402.07003}.

\bibitem[Hut11]{qech}
Michael Hutchings.
\newblock Quantitative embedded contact homology.
\newblock {\em J. Differential Geom.}, 88(2):231--266, 2011.

\bibitem[HZ11]{HZ}
Helmut Hofer and Eduard Zehnder.
\newblock {\em Symplectic Invariants and Hamiltonian Dynamics}.
\newblock Springer Basel, 2011.

\bibitem[Sch18]{schlenk}
Felix Schlenk.
\newblock Symplectic embedding problems, old and new.
\newblock {\em Bull. Amer. Math. Soc.}, 55:139--182, 2018.

\bibitem[Tra95]{traynor}
Lisa Traynor.
\newblock Symplectic packing constructions.
\newblock {\em J. Differential Geom.}, 42(2):411--429, 1995.

\bibitem[Vit99]{V}
Claude Viterbo.
\newblock Functors and computations in {F}loer homology with applications. {I}.
\newblock {\em Geom. Funct. Anal.}, 9(5):985--1033, 1999.

\end{thebibliography}

\end{document}